\documentclass{amsart}

\usepackage{amsaddr}

\usepackage[T1]{fontenc}
\usepackage[dvips]{graphicx}
\usepackage{epsfig}
\usepackage{amsmath,amssymb,amscd,amsthm}

\usepackage[latin1]{inputenc}
\usepackage{latexsym}
\usepackage{graphics}
\usepackage{picins}
\usepackage{colortbl}
\usepackage{color}
\usepackage{ae}
\usepackage[cm]{aeguill}

\def\Tor{{\mathbb T}}
\def\dist{\mathrm{dist\,}}
\def\diam{\mathrm{diam\,}}
\usepackage[cspex,bbgreekl]{mathbbol}
\graphicspath{{Figures/}}
\def\H{{\mathbb H}}

\usepackage{textcomp}
\begin{document}

\setcounter{secnumdepth}{3}
\setcounter{tocdepth}{2}
\newtheorem{theorem}{Theorem}[section]
\newtheorem{definition}[theorem]{Definition}
\newtheorem{lemma}[theorem]{Lemma}
\newtheorem{proposition}[theorem]{Proposition}
\newtheorem{corollary}[theorem]{Corollary}
\newtheorem{gstatement}{Problem}
\newcommand{\mf}{\mathfrak}
\newcommand{\mb}{\mathbb}
\newcommand{\ol}{\overline}
\newcommand{\la}{\langle}
\newcommand{\ra}{\rangle}

\newtheorem{Alphatheorem}{Theorem}
\renewcommand{\theAlphatheorem}{\Alph{Alphatheorem}} 

\newtheorem{Alphatheoremprime}{Theorem}
\renewcommand{\theAlphatheoremprime}{\theAlphatheorem\textquotesingle}

\newcommand{\EM}{\ensuremath}
\newcommand{\norm}[1]{\EM{\left\| #1 \right\|}}

\newcommand{\modul}[1]{\left| #1\right|}

\def\N{{\mathbb N}}
\def\R{{\mathbb R}}
\def\Z{{\mathbb Z}}
\def\Sph{{\mathbb S}}
\def\H{{\mathbb H}}
\def\phi{\varphi}
\def\epsilon{\varepsilon}
\def\V{{\mathcal V}}
\def\E{{\mathcal E}}
\def\F{{\mathcal F}}
\def\C{{\mathcal C}}
\def\inn{\mathrm{int}}
\def\grad{\mathrm{grad\,}}
\def\ext{\mathrm{ext}}
\def\Vol{\mathrm{Vol}}

\def\M{{\mathcal M}}
\def\T{{\mathcal T}}

\title{Hyperbolic cusps with convex polyhedral boundary}
\author{Fran\c{c}ois Fillastre}
\address{Department of Mathematics \\
 University of Fribourg, P\'erolles \\
 Chemin du Mus\'ee 23 \\
 CH-1700 Fribourg  \\
 SWITZERLAND}
\email{francois.fillastre@unifr.ch}
\thanks{The first author was partially supported by Schweizerischer Nationalfonds 200020-113199/1} 

 \author{Ivan Izmestiev}
\thanks{The second author was supported by the DFG Research Unit 565 ``Polyhedral Surfaces''}
\address{Institut f\"ur Mathematik, MA 8-3 \\
Technische Universit\"at Berlin \\
Str. des 17. Juni 136 \\
D-10623 Berlin \\
 GERMANY}
\email{izmestiev@math.tu-Berlin.de}

\date{\today}
\maketitle

\begin{abstract}
We prove that a 3-dimensional hyperbolic cusp with convex polyhedral boundary is uniquely determined by the metric induced on its boundary. Furthemore, any hyperbolic metric on the torus with cone singularities of positive curvature can be realized as the induced metric on the boundary of a convex polyhedral cusp.

The proof uses the total scalar curvature functional on the space of ``cusps with particles'', which are hyperbolic cone-manifolds with the singular locus a union of half-lines. We prove, in addition, that convex polyhedral cusps with particles are rigid with respect to the induced metric on the boundary and the curvatures of the singular locus.

Our main theorem is equivalent to a part of a general statement about isometric immersions of compact surfaces.
\end{abstract}


\section{Introduction}

In Subsection \ref{subsec:Statements} we state the results, Subsection \ref{sub:related works} puts them in a more general context, and Subsection \ref{sketch} gives a sketch of the proof and a plan of the paper. Precise definitions will be given in Section \ref{Definitions}.

\subsection{Statements}\label{subsec:Statements}

Let $M \approx \Tor \times [0,+\infty)$ be a convex hyperbolic 3-manifold
with a cusp and with piecewise geodesic boundary. We often call it a \emph{convex polyhedral cusp}. The induced metric on~$\partial M$ is a hyperbolic metric on the torus~$\Tor$ with conical singularities of positive singular curvature.  The main result of this paper is that the metric on $M$ is uniquely determined by the metric on $\partial M$:

\begin{Alphatheorem}\label{Base theorem}
Let $g$ be a hyperbolic metric with conical singularities of positive singular curvature on the $2$-torus $\Tor$. Then there exists a convex polyhedral cusp $M$ such that $\partial M$ with the induced metric is isometric to $(\Tor, g)$. Furthermore, $M$ is unique up to isometry.
\end{Alphatheorem}

This theorem can be viewed as a statement about isometric immersions. A \emph{convex parabolic polyhedron} is a pair $(P,G)$, where $P$ is a convex polyhedron in $\mathbb{H}^3$, and $G$ is a discrete subgroup of $\mathrm{Iso}^+(\mathbb{H}^3)$ that acts freely cocompactly on a horosphere and leaves $P$ invariant. Figure \ref{fig:ellipsoid} shows an example of a convex parabolic polyhedron, whose vertices form an orbit of the group $G$. For any convex parabolic polyhedron $(P,G)$, the quotient $P/G$ is a convex polyhedral cusp. Conversely, the universal cover of a convex polyhedral cusp is isometric to a convex parabolic polyhedron. Thus Theorem \ref{Base theorem} says that each hyperbolic metric on $\Tor$ with conical singularities of positive singular curvature can be uniquely realized as the boundary of a convex parabolic polyhedron:

\begin{Alphatheoremprime}\label{realisation theorem}
Let $g$ be a hyperbolic metric with conical singularities of positive singular curvature on the torus $\Tor$. Then there exists a unique up to equivariant isometry convex parabolic polyhedron $(P,G)$ such that $\partial P/G$ is isometric to $(\Tor,g)$.
\end{Alphatheoremprime}

This is a part of a general statement about polyhedral realization of metrics on compact surfaces, see Subsection \ref{sub:related works}.

\begin{figure}[ht]
\includegraphics{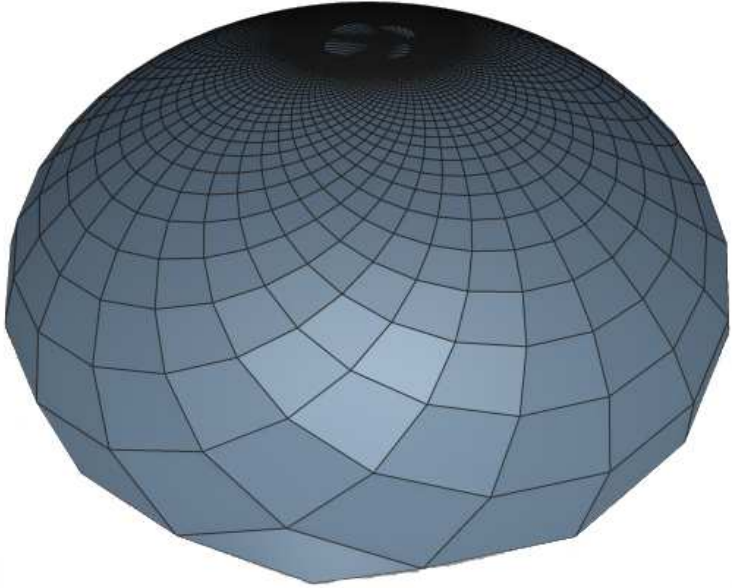}
\caption{A simplest convex parabolic polyhedron in the Klein projective model.}
\label{fig:ellipsoid}
\end{figure}
The uniqueness part of Theorem \ref{Base theorem} is a rigidity statement: two convex polyhedral cusps with isometric boundaries are isometric. Compare this with the Cauchy-Alexandrov theorem on rigidity of convex polytopes, \cite{Cauchy13}, \cite{AlexandrovBook}. We prove also the corresponding infinitesimal rigidity result: any non-trivial first-order deformation of the metric on $M$ in the class of complete hyperbolic metrics induces a non-trivial first-order deformation of the metric on $\partial M$.

\begin{Alphatheorem}\label{InfRigCusp}
Convex polyhedral cusps are infinitesimally rigid.
\end{Alphatheorem}

\begin{Alphatheoremprime}\label{infinitesimal rigidity}
Convex parabolic polyhedra are parabolically infinitesimally rigid.
\end{Alphatheoremprime}
For definition of parabolic infinitesimal rigidity see Subsection \ref{sub:proofs rigidity}.

Our method involves study of \emph{convex polyhedral cusps with particles}. These are cone-manifolds that are very much like convex polyhedral cusps but have cone singularities along half-lines (the particles) that start at the cone singularities of the boundary.

For cusps with particles we prove a global rigidity statement:

\begin{Alphatheorem}\label{cusps w particles determined by curvatures}
Two convex polyhedral cusps with particles with the same metric on the boundary and the same singular curvatures are isometric.
\end{Alphatheorem}

\subsection{Related work}\label{sub:related works}

\subsubsection{Towards a general realization statement.}

Theorem \ref{realisation theorem} is similar to a famous theorem of A.D. Alexandrov:
\begin{theorem}[A.~D.~Alexandrov, \cite{alex42, AlexandrovBook}] \label{AlexThm}
Let $g$ be a metric of constant curvature $K$ with conical singularities of positive singular curvature on the $2$-sphere~$\Sph$. Then $(\Sph,g)$ can be realized as a convex polyhedral surface in the $3$-dimensional Riemannian space-form of curvature $K$. The realization is unique up to an ambient isometry.
\end{theorem}
Clearly, the positivity condition on the singular curvatures is necessary if one wants to realize the given metric as a convex polyhedral surface in a Riemannian space-form. In \emph{Lorentzian space-forms}, convex space-like polyhedral surfaces have most often singularities of \emph{negative} singular curvature, \cite{schlorentz,Fillastre2}.
\begin{theorem}[Rivin, Rivin--Hodgson, \cite{rivinthese,RivHod}] \label{RivHod}
Let $g$ be a spherical metric with negative cone singularities on $\Sph$ and lengths of closed geodesics greater than $2\pi$. Then $(\Sph,g)$ can be uniquely realized as a convex polyhedral surface in de Sitter space.
\end{theorem}
Actually, the uniqueness statement proved in \cite{RivHod} is slightly weaker, see \cite{schlorentz}.

Realization theorems are proved for compact surfaces of genus $\geq 2$ in hyperbolic space \cite{Fillastre1} and in Lorentzian space-forms \cite{Schpoly,Fillastre2}. As a matter of fact, only one case of constant curvature metric with conical singularities of constant sign on compact surfaces has not been treated yet: that of metrics on the torus
which can be realized in de Sitter space. This is the subject of \cite{andreevtorus} that uses the same method as the present paper.
Putting all together would lead to a solution of the following problem.

\begin{gstatement}
Let $g$ be a metric of constant curvature $K$ with conical singularities of a constant sign $\epsilon \in \{-,+\}$ on a compact surface $S$. In the case $K=1,\epsilon=-$ we require the lengths of contractible geodesics to be $>2\pi$. Then the universal cover of $(S, g)$ can be uniquely realized in $M_K^{\epsilon}$ as a convex polyhedral surface invariant under the action of a representation of $\pi_1(S)$ in a $3$-dimensional subgroup of $\mathrm{Iso}^+(M_K^{\epsilon})$.
\end{gstatement}

Here $M_K^+$ is the Riemannian space-form of curvature $K$, and $M_K^-$ is the Lorentzian space-form of curvature $K$.

A more traditional way to state Theorem \ref{RivHod} is in terms of the dual metric of a convex hyperbolic polyhedron, which is obtained with the help of the Gauss map, \cite{RivHod}. If the combinatorics of the polyhedron is known, then the dual metric is defined by the values of the dihedral angles. This implies Andreev's Theorem \cite{And70} about compact acute-angled hyperbolic polytopes. Similarly, realization theorems for higher genus in de Sitter space \cite{Schpoly,Fillastre2,andreevtorus} imply existence and uniqueness of circle patterns with acute (exterior) intersection angles between the circles. A more thorough discussion will be given in \cite{andreevtorus}.

\subsubsection{Hyperbolic manifolds with convex polyhedral boundary.}

Here we restrict our attention to the hyperbolic cases of theorems above. A reformulation of Theorem \ref{AlexThm} is that each hyperbolic cone metric on the sphere with singularities of positive curvature can be uniquely extended to a hyperbolic metric with a convex polyhedral boundary on the ball.

In the same way, hyperbolic realization theorem for genus~$\geq 2$ \cite{Fillastre1} says that the metric inside a ``Fuchsian manifold'' with convex polyhedral boundary is uniquely determined by the metric on the boundary. Both are special cases of the following statement.

\begin{gstatement} \label{g:AB}
Let $M$ be a compact connected $3$-manifold with boundary, and let $M$ admit a complete hyperbolic convex cocompact metric. Then each hyperbolic cone metric on $\partial M$ with singularities of positive curvature can be uniquely extended to a hyperbolic metric on $M$ with convex polyhedral boundary.
\end{gstatement}

In the case of smooth strictly convex boundary the analog was proved in \cite{Schconvex} (the case of the ball should follow from the works of Alexandrov and Pogorelov). In both polyhedral and smooth cases the same problem can be posed for geometrically finite manifolds. Theorem \ref{Base theorem} provides the simplest polyhedral case of such generalization.
Similar questions can be posed about the dual metric on the boundary. In the smooth compact case the dual metric is simply the third fundamental form, and the problem is solved also in \cite{schconvlor,SchLab,Schconvex}.

\subsubsection{Manifolds with particles.}

The term ``manifold with particles'' comes from the physics literature, where the manifolds are Lorentzian and the singularities are along time-like geodesics. The definition can be naturally extended to certain hyperbolic cone-manifolds with singularities along infinite lines, see \emph{e.g.} \cite{schMinparticles,schparticlesfuchs,schparticlesAdS}. By analogy, we have adopted the same terminology for our ``cusps with particles''.

Theorem \ref{cusps w particles determined by curvatures} states that a convex polyhedral cusp with particles is uniquely determined by the metric on its boundary and the singular curvatures along the particles. One can ask what boundary metric and particles curvatures can be realized.

\begin{gstatement}
Let $g$ be a hyperbolic cone metric on $\Tor$ with $n$ singularities of positive curvature.  What are the necessary and sufficient conditions on the numbers $\kappa_1, \ldots, \kappa_n$ so that there exists a cusp with particles of curvatures $\kappa_1, \ldots, \kappa_n$ and with convex polyhedral boundary isometric to $g$?
\end{gstatement}

One obvious condition on $(\kappa_i)$ is $\sum_{i=1}^n \kappa_i = 0$, see Lemma \ref{Lem:sumcurvaturearezero}.

\subsubsection{Weakly convex star-shaped parabolic polyhedra.}\label{sub:weakly convex}

A \emph{star-shaped parabolic polyhedron} is a pair $(P,G)$, where $P \subset \H^3$ is the cone with the apex $c \in \partial \overline{\H^3}$ over a polyhedral surface that projects bijectively onto horospheres with center $c$, and $G$ is a discrete subgroup of $\mathrm{Iso}^+(\mathbb{H}^3)$ that acts freely cocompactly on horospheres with center $c$ and leaves $P$ invariant. Clearly, every convex parabolic polyhedron is star-shaped, but the converse does not hold.

A star-shaped parabolic polyhedron is called \emph{weakly convex}, if its vertices are vertices of some convex polyhedron.

By using the argument from \cite{schweakconvex}, we prove the following theorem.

\begin{Alphatheorem}\label{thm:weak infinitesimal rigidity}
Weakly convex star-shaped parabolic polyhedra  are parabolic infinitesimally rigid.
\end{Alphatheorem}

\subsection{Sketch of the proof and plan of the paper}\label{sketch}

We prove Theorem \ref{Base theorem} by the variational method. The variational method consists in identifying the object we are looking for with a critical point of a functional. If a concave functional on a convex domain attains its maximum in the interior, then the maximum point is the unique critical point. This yields both the existence and uniqueness statement for the desired object.

The domain that we consider is the space $\M(\Tor, g)$ of convex polyhedral cusps with particles and with boundary $(\Tor, g)$. A cusp with particles is glued from semi-ideal pyramids with the common ideal apex so that the pyramids allow a consistent \emph{truncation} by horospheres.
A truncation yields a collection $(h_i)_{i \in \Sigma}$ of truncated particle lengths, one for each singular point $i \in \Sigma$. Change of a truncation results in adding a common constant to all of the $h_i$. We call the corresponding equivalence class $[h]$ the \emph{particle lengths}.

Section \ref{sec:SpaceOfCusps} contains two important results. First, we show that for a given metric $g$ on the boundary, a convex polyhedral cusp with particles is uniquely determined by its particle lengths $[h]$. That is to say, there don't exist two convex cusps with particles with different face structures and same particle lengths. Second, we show that $\M(\Tor, g)$ is a compact convex subset of $\R^\Sigma / \la \mathbb{1} \ra$.

The functional on $\M(\Tor,g)$ is given by the formula
$$
S(M) = -2 \Vol(M) + \sum h_i \kappa_i + \sum \ell_e(\pi - \theta_e).
$$
Here the first sum ranges over all singularities of the metric $g$, and $\kappa_i$ denotes the singular curvature at the $i$th particle. The sum does not depend on the choice of a truncation due to $\sum \kappa_i = 0$. The second sum is of a similar nature: here $\ell_e$ is the length of a boundary edge $e$, and $\theta_e$ is the dihedral angle at this edge. Functional $S$ is the discrete analog of the total scalar curvature, which is also known as the Hilbert-Einstein functional.

Schl\"afli's formula implies
$$
\frac{\partial S}{\partial h_i} = \kappa_i.
$$
Thus a critical point of $S$ corresponds to a convex polyhedral cusp with vanishing curvatures of particles. An explicit computation of derivatives shows that the Hessian of $S$ is negatively semidefinite. Although at some points the Hessian might be degenerate, it turns out that the functional $S$ is strictly concave on $\M(\Tor, g)$. Functional $S$ is investigated in Section \ref{sec:scalar curvature}.

Proofs of Theorems \ref{Base theorem} --- \ref{thm:weak infinitesimal rigidity} are given is Section \ref{sec:proofs}. All of them use either the non-degeneracy of the Hessian or the strict concavity of $S$.

\subsection{Remarks}
In the physics literature, the functional $\sum h_i \kappa_i$ for a manifold built up from Euclidean simplices is known as the \emph{Regge functional}. In mathematics, the boundary term $\sum \ell_e (\pi - \theta_e)$ appeared in the works of Steiner and Minkowski. Minkowski also showed that this is the correct discrete analog of the total mean curvature of the boundary of a convex body. In the smooth case, Blaschke and Herglotz \cite{BH37} suggested to use the Hilbert-Einstein functional to approach Weyl's problem, which is a smooth analog of Alexandrov's theorem in $\R^3$: show that any convex Riemannian metric on the sphere is uniquely realized as the boundary of a convex body. Recently, Michael Anderson \cite{And02} proposed an approach to the geometrization of 3-manifolds via scalar curvature type functionals.

The variational method used in the present paper was earlier applied in \cite{Izmestiev} to prove the existence and uniqueness of a Euclidean convex cap with given metric on the boundary. Functional $S$ was also used in \cite{BobenkoIzmestiev} to give a new proof of Alexandrov's theorem in $\R^3$. In \cite{BobenkoIzmestiev}, the matter was complicated by the fact that $S$ was neither concave nor convex.

An alternative method of proving realization statements like Theorems \ref{AlexThm}, \ref{RivHod}, see also \cite{Fillastre1, Fillastre2}, is the deformation method, also known as Alexandrov's method. The idea is to consider the map between the space of convex polyhedral surfaces and the space of cone metrics that associates to a surface its induced metric. The key point is to prove the local rigidity: a deformation of a surface always induces a deformation of a metric. In other words, the map ``induced metric'' is a local homeomorphism. Then, by topological arguments, this map is shown to be a global homeomorphism. Note a different role of the infinitesimal rigidity in the two approaches. Being a key lemma in the deformation method, it is a byproduct in the variational method (non-degeneracy of the Hessian at a critical point).

The variational method is constructive: a computer program can be written that finds the critical point of a functional numerically. For Alexandrov's theorem in $\R^3$, such a program was created by Stefan Sechelmann and is available at {\tt http://www.math.tu-berlin.de/geometrie/ps/software.shtml}.

\subsection{Acknowledgments}

Both authors want to thank Cyril Lecuire and Jean-Marc Schlenker for useful conversations, as well as Stefan Sechelmann who made Figure \ref{fig:ellipsoid}.


\section{Definitions and preliminaries}\label{Definitions}
In Subsection \ref{subsec:Cusps} we convex polyhedral cusps. These are hyperbolic cusps whose metric in the neighborhood of boundary points is modelled on convex polyhedral cones. A \emph{convex polyhedral cone} is the intersection of finitely many halfspaces in $\H^3$ whose boundary planes pass through one point. Then we define convex parabolic polyhedra and show that they are universal covers of convex polyhedral cusps. In Subsection \ref{subsec:CuspsParticles} we define hyperbolic cusps with particles as cone-manifolds glued from semi-ideal pyramids. Finally, Subsection \ref{sub:background} contains some hyperbolic geometry needed in the sequel.

\subsection{Cusps and parabolic polyhedra.} \label{subsec:Cusps}

\begin{definition} \label{def:Cusp}
A \emph{hyperbolic cusp with boundary} is a complete hyperbolic manifold of finite volume homeomorphic to $\Tor \times [0, +\infty)$. We say that the cusp has a \emph{convex polyhedral boundary} if every point on the boundary has a neighborhood isometric to a neighborhood of a point on the boundary of a convex polyhedral cone in $\H^3$.
\end{definition}
We call hyperbolic cusps with convex polyhedral boundary briefly \emph{convex polyhedral cusps}.
Clearly, the induced metric on the boundary of a convex polyhedral cusp is a hyperbolic metric with cone singularities of positive curvature. It is easy to define vertices, edges and faces of a convex polyhedral cusp. Vertices are exactly the cone singularities of the metric on the boundary. Every edge is a geodesic joining the vertices. Edges cut the boundary $\partial M$ of the cusp $M$ into faces, which are maximal connected open subsets of $\partial M$ that bound $M$ geodesically.

\begin{definition}
A \emph{convex parabolic polyhedron} in $\H^3$ is a pair $(P,G)$, where $P \subset \H^3$ is the convex hull of a discrete set of points, and $G$ is a discrete subgroup of $\mathrm{Iso}^+(\H^3)$ that acts freely cocompactly on a horosphere in $\H^3$ and leaves $P$ invariant.
\end{definition}
Clearly, the vertex set of $P$ is $G$-invariant. Since it is discrete, it is the union of finitely many orbits of the group $G$. The simplest example of a convex parabolic polyhedron is the convex hull of one orbit, see Figure \ref{fig:ellipsoid}.

The group $G$ has a unique fixed point $c$ in $\partial \overline{\H^3}$. Clearly, $c$ lies in the closure of $P$. We call $c$ the \emph{center} of the polyhedron $P$.

\begin{lemma} \label{lem:PolToCusp}
Let $(P,G)$ be a convex parabolic polyhedron. Then the quotient space $P/G$ is a convex polyhedral cusp.
\end{lemma}
\begin{proof}
It is immediate that $G \cong \Z^2$ and $B/G \approx \Tor \times [0,+\infty)$ for any horoball $B$ centered at $c$. Since the vertex set of $P$ is the union of finitely many orbits, there are horoballs $B_1$ and $B_2$ centered at $c$ such that $B_1 \subset P \subset B_2$. It is easy to see that any geodesic passing through $c$ intersects the boundary of $P$ at exactly one point. It follows that $P/G$ is homeomorphic to $\Tor\times[0,+\infty)$. From $P \subset B_2$ it also follows that $P/G$ has finite volume. The manifold $P/G$ is complete since it is a closed subset of a complete manifold $B_2/G$. Finally, $P/G$ has convex polyhedral boundary because $P$ has.
\end{proof}

Let $M$ be a convex polyhedral cusp. By definition it is locally convex, hence it is convex, \cite[Corollary I.1.3.7.]{Notesonnotes}. It follows that the developing map
$D: \widetilde{M} \to \H^3$ is an isometric embedding, \cite[Proposition I.1.4.2.]{Notesonnotes}. The action of the fundamental group $\pi_1M \cong \Z^2$ on $\widetilde{M}$ by deck transformations yields a representation $\rho: \pi_1M \to \mathrm{Iso}^+(\H^3)$.

\begin{lemma} \label{lem:CuspToPol}
The pair $(D(\widetilde{M}), \rho(\pi_1M))$ is a convex parabolic polyhedron.
\end{lemma}
\begin{proof}
Clearly, $D(\widetilde{M})$ is a convex polyhedron homeomorphic to the half-space. Its vertices form a discrete set, because they correspond to vertices of $M$, whose number is finite.

The thin part of $M$ contains a totally umbilic torus $C$ with Euclidean metric. It follows that the developing map maps the universal cover of $C$ to  a horosphere. The group $\rho(\pi_1M)$ acts on $D(\widetilde{C})$ freely with a compact orbit space $C$. The lemma follows.
\end{proof}
\begin{corollary}
Every face of a convex polyhedral cusp is a convex hyperbolic polygon.
\end{corollary}

Lemmas \ref{lem:PolToCusp} and \ref{lem:CuspToPol} imply that $(P,G) \mapsto P/G$ is a one-to-one correspondence between the equivariant isometry classes of convex parabolic polyhedra and isometry classes of convex polyhedral cusps. Thus Theorem \ref{Base theorem} is equivalent to Theorem~\ref{realisation theorem}.

\subsection{Cusps with particles.} \label{subsec:CuspsParticles}
\begin{definition}
A \emph{semi-ideal pyramid} in $\H^3$ is the convex hull of a convex polygon $A$ and a point $a \in \partial\overline{\H^3}$ such that $a$ is not coplanar to $A$. The point $a$ is called the \emph{apex} of the pyramid, the polygon $A$ its \emph{base}.
\end{definition}
A convex polyhedral cusp can be decomposed into semi-ideal pyramids with a common apex. Indeed, let $M$ be a cusp and let $(P,G)$ be the corresponding parabolic polyhedron. If $c \in \partial\overline{\H^3}$ is the center of $P$, then $P$ is composed from semi-ideal pyramids with the apex $c$ over the faces of $P$. Clearly, this decomposition of $P$ descends to a decomposition of $M \cong P/G$. In the example on Figure \ref{fig:ellipsoid}, the decomposition of $M$ consists of a single isosceles quadrangular pyramid whose faces are identified according to the standard gluing of a torus from a parallelogram.

Let us see when a gluing of pyramids defines a convex polyhedral cusp.
\begin{definition}
A \emph{cuspidal complex} is a collection of semi-ideal pyramids glued isometrically along some pairs of faces so that combinatorially the gluing is represented by the cone with an ideal apex over a polyhedral decomposition of the torus.
\end{definition}
If the pyramids of a cuspidal complex fit well around their lateral edges, then the result of the gluing is a hyperbolic manifold with polyhedral boundary. This manifold can be non-complete as the following example shows. In the Poincar\'e half-space model, take a semi-ideal pyramid with vertices $(1,0,1), (0,1,1), (2,0,2), (0,2,2)$ and the point at infinity as the apex. Clearly, the semi-ideal triangles in each pair of opposite sides of the pyramid are isometric. When we identify them, we get a non-complete manifold homeomorphic to $\Tor\times[0,+\infty)$.

For a semi-ideal pyramid $\Delta$, choose a horoball $B$ centered at the apex of the pyramid and disjoint with its base. The body $\Delta\setminus B$ is called a \emph{truncated semi-ideal pyramid} or a \emph{horoprism}.
\begin{definition}
A cuspidal complex is called \emph{compatible} if every pyramid of the complex can be truncated so that the gluing isometries restrict to the faces of the truncated pyramids.
\end{definition}

\begin{lemma}
The manifold defined by a cuspidal complex is complete if and only if the complex is compatible.
\end{lemma}
\begin{proof}
If the manifold is complete, then its thin part contains a cusp with totally umbilic Euclidean boundary. Cutting this cusp off defines a compatible truncation of the pyramids.

Conversely, assume that the complex is compatible. For every pyramid $\Delta_i$ consider the corresponding horoball sectors $\Delta_i \cap B_i$. It is easy to see that the developing map maps the union $\cup_i (\Delta_i \cap B_i)$ to a horoball in $\H^3$. Thus the manifold $\cup_i (\Delta_i \cap B_i)$ is complete and so is the whole manifold defined by the complex.
\end{proof}

In general, a compatible cuspidal complex defines a cone-manifold whose singular locus is contained in the union of half-lines that come from the lateral edges of the semi-ideal pyramids. We call these half-lines \emph{particles}.
\begin{definition}
A \emph{polyhedral cusp with particles} is a hyperbolic cone-manifold defined by a compatible cuspidal complex. A polyhedral cusp with particles is called \emph{convex} iff the total dihedral angle at every boundary edge is $\le \pi$. For a boundary singularity $i$, denote by $\omega_i$ the total dihedral angle around the $i$th particle. The \emph{curvature} of the $i$th particle is defined as
$$
\kappa_i = 2\pi - \omega_i.
$$

A \emph{truncated polyhedral cusp with particles} is defined in the same way as a polyhedral cusp with particles, using horoprisms instead of semi-ideal pyramids.
\end{definition}

We are interested in the cusps whose boundary is isometric to $(\Tor, g)$, where $g$ 
is a hyperbolic metric with conical singularities of positive singular curvatures on the torus $\mathbb{T}$.

\begin{definition}
We denote by $\M(\Tor,g)$ the space of convex polyhedral cusps with particles $M$ with $\partial M = (\Tor,g)$. By $\M_{tr}(\Tor,g)$ we denote the space of truncated convex polyhedral cusps with particles with the boundary $(\Tor,g)$.
\end{definition}
Formally speaking, an element of $\M(\Tor, g)$ is a pair $(M,f)$, where $f: \partial M \to (\Tor, g)$ is an isometry. It will be convenient to us to identify $\partial M$ with the given metric torus $(\Tor, g)$, so that we can omit mentioning $f$.

Theorem \ref{Base theorem} is equivalent to say that in $\M(\Tor, g)$ there is a unique cusp with vanishing curvatures of particles. Note that we don't fix an isometry between $\partial M$ and $(\Tor,g)$ in Theorem \ref{Base theorem}. In this case, it does not really matter because the uniqueness is stated.

As in the case of a convex polyhedral cusp, the boundary of a convex polyhedral cusp with particles consists of vertices, edges and faces. Unlike the case without particles, faces of a cusp with particles can be non-simply connected, and there can be isolated vertices, as the following example shows.

\medskip

\noindent\emph{Example}. In the upper half-space, take the point $a=(0,0,1)$ and points $b, c$ on the unit sphere centered at $0$ so that $b$ and $c$ lie at an equal distance from $a$ and the angle at the vertex $a$ in the spherical triangle $abc$ is $<\frac{\pi}{2}$. In the Poincar\'e half-space model, the semi-ideal pyramid with the base $abc$ and the apex at the point at infinity has dihedral angles $\frac{\pi}{2}$ at the edges $ab$ and $ac$ and an angle $<\frac{\pi}{2}$ at $bc$. Take four copies of this pyramid and glue them cyclically around the edge $a\infty$. The result is a semi-ideal quadrangular pyramid with a particle. By identifying the pairs of its opposite sides, we obtain a convex polyhedral cusp with particles. Its boundary contains two vertices, two loop edges and a single face that looks as a punctured square.

\medskip

To deal with the space $\M(\Tor, g)$, we need to introduce coordinates on it. A compatible cuspidal complex over $(\Tor,g)$ is determined by a polyhedral subdivision of the metric torus $(\Tor, g)$ and by the lengths of the lateral edges of the pyramids over the faces of the subdivision. To measure the lengths of the (infinite) lateral edges in a compatible complex, one chooses a truncation and measures the lengths of the truncated edges. A different choice of truncation results in adding a constant to all of the lengths. As for the polyhedral subdivision of the torus, it is convenient to refine it to a triangulation. This motivates the following definition.
\begin{definition} \label{dfn:Th}
Let $M_{tr} \in \M_{tr}(\Tor,g)$ be a truncated convex polyhedral cusp with particles. Let $T$ be a triangulation of $(\Tor,g)$ that refines the natural decomposition of the boundary $\partial M_{tr}$, and let $h_i$ be the length of the truncated particle with the endpoint $i \in \Sigma$, where $\Sigma$ is the set of singularities of $g$. We associate to $M_{tr}$ the pair $(T,h)$, where $h$ stands for $(h_i)_{i\in \Sigma}$.

Similarly, to every $M \in \M(\Tor,g)$ we associate a pair $(T,[h])$, where $(T,h)$ represents a truncation of $M$, and $[h]$ is the equivalence class under the relation $h \sim h' \Leftrightarrow h'_i = h_i + c$ for all $i$ and some constant $c$.

The equivalent class $[h]$ is called the \emph{particle lengths} of $M$. The cusp $M$ is \emph{isosceles} if it is made of isosceles semi-ideal pyramids, \emph{i.e.} if $[h]=[0,\ldots,0]$.
\end{definition}
By a triangulation we mean a decomposition of $(\Tor,g)$ into open hyperbolic triangles by geodesic arcs (edges of the triangulation) with endpoints in $\Sigma$. We don't impose any restrictions on the combinatorics, so that there may be loops and multiple edges, and two triangles may have two edges in common, and two edges of a triangle may be identified. An edge with endpoints $i$ and $j$ is denoted by $ij$, a triangle with vertices $i, j$ and $k$ is denoted by $ijk$. Because of what we just said, different edges or triangles may obtain the same notation, and some letters in the notation may repeat. But this will not lead to confusion.

\begin{lemma}\label{Lem:sumcurvaturearezero}
The curvatures of a convex polyhedral cusp with particles satisfy
$$ \sum_{i\in\Sigma} \kappa_i=0.$$
\end{lemma}
\begin{proof}
Truncate the cusp. The induced metric on the surface of truncation is a flat metric with conical singularities on the torus. Clearly, the curvatures of the singularities are exactly the $\kappa_i$. The lemma follows from the Gauss--Bonnet formula.
\end{proof}

\subsection{Some hyperbolic trigonometry}\label{sub:background}
\begin{lemma}[Cosine law for semi-ideal triangles] \label{horosphere lemma}
Let $B$ be a horodisk in the hyperbolic plane and let $i, j$ be two points not in $B$. Let $h_i, h_j, \lambda$ be the distances $\dist(i,B)$, $\dist(j,B)$, $\dist(i,j)$, respectively, and let $\rho_i$ be the angle between the geodesic segment $ij$ and the perpendicular from $i$ to $B$. Then
\begin{equation}\label{eqn:cos}
\cos \rho_i=\frac{\cosh\lambda-e^{h_j-h_i}}{\sinh\lambda}.
\end{equation}
\end{lemma}
\begin{proof}
Go to the limit in the cosine law for the triangle with vertices $i,j$ and third vertex approaching the center of the horodisk.
\end{proof}

\begin{lemma} \label{lemmadistancelike}
Let $B$ be a horodisk in the hyperbolic plane, and $L$ be a line disjoint with $B$. Then for every $x \in L$ we have
$$
\dist(x,B) = \log \cosh (\dist(x,a)) + \dist(a,B),
$$
where $a$ is the point on $L$ nearest to $B$.
\end{lemma}
\begin{proof}
Consider the semi-ideal triangle with vertices $a, x$ and the center $c(B)$ of $B$. Apply equation (\ref{eqn:cos}), taking $a$ for $i$, $x$ for $j$. We have $\cos \rho_i = 0$ and $\lambda = \dist(x,a)$. Hence
$$
\cosh (\dist(x,a)) = e^{h_j-h_i} = \exp(\dist(x,B) - \dist(a,B)),
$$
and the claim follows.
\end{proof}

\begin{lemma} \label{lem:SinhForm}
Let $i,j,k$ be three collinear points in $\H^2$ such that $j$ lies between $i$ and $k$. For a horodisk $B$ that contains none of the points $i,j,k$, denote by $h_i, h_j, h_k$ the distances $\dist(i,B), \dist(j,B), \dist(k,B),$ respectively. Then
\begin{equation} \label{eqn:SinhForm}
e^{h_j} = \frac{\sinh\mu}{\sinh(\lambda+\mu)} e^{h_i} + \frac{\sinh\lambda}{\sinh(\lambda+\mu)} e^{h_k},
\end{equation}
where $\lambda = \dist(i,j), \mu = \dist(j,k)$.
\end{lemma}
\begin{proof}
Let $c$ be the center of $B$, and let $\rho_j, \pi-\rho_j$ be the angles between the geodesic $ik$ and the perpendicular from $j$ to $\partial B$. Compute $\cos \rho_j$ by the equation (\ref{eqn:cos}) from the semi-ideal triangles $ijc$ and $jkc$ and equate the two expressions.
\end{proof}

\section{The space of convex polyhedral cusps with particles} \label{sec:SpaceOfCusps}
In Subsection \ref{subsec:HeightsToCusp} we show that a convex polyhedral cusp with particles is uniquely determined by the particle lengths $[h]$ introduced in Definition \ref{dfn:Th}. This identifies the space of cusps $\M(\Tor,g)$ with a subset of $\R^{n-1}$, where $n$ is the number of singularities of $g$. Subsection \ref{subsec:Lemmas} contains several lemmas that are later used in Subsection \ref{subsec:MTg} to prove Proposition \ref{prp:SpaceOfCusps}. The proposition says that $\M(\Tor,g)$ is a compact convex subset of $\R^{n-1}$.

Everywhere in this section we mean by a \emph{cusp} a cusp with particles with polyhedral boundary.

\subsection{Particle lengths define a cusp} \label{subsec:HeightsToCusp}
Recall that a truncated convex cusp was defined as a union of \emph{horoprisms}. A horoprism is a semi-ideal pyramid with a neighborhood of the ideal vertex cut off along a horosphere. A horoprism has a \emph{hyperbolic base} and a \emph{Euclidean base}. The lateral edges of a horoprism are orthogonal to its Euclidean base. The lengths of lateral edges are called \emph{heights} of the horoprism. Clearly, a horoprism is uniquely determined by the hyperbolic base and the heights. In what follows, we consider only triangular horoprisms.

Cutting a truncated convex cusp into triangular horoprisms produces a pair $(T,h)$, where $T$ is a geodesic triangulation of $(\Tor,g)$, and $h=(h_i)_{i\in\Sigma}$ is the collection of heights of the horoprisms, which at the same time are the truncated particle lengths in the cusp. Occasionally, there is some freedom in the choice of $T$, since it may be any refinement of the canonical face decomposition of the cusp boundary.

Our goal is to prove
\begin{proposition} \label{prp:HeightsToCusp}
A truncated convex cusp is uniquely determined by its truncated particle lengths.
\end{proposition}
In other words, if $(T, h)$ and $(T', h)$ are pairs associated with the truncated convex cusps $M_{tr}, M'_{tr}$, respectively, then $M_{tr} = M'_{tr}$.

The following definition introduces a concept that will be used through the whole Section.
\begin{definition}
The \emph{distance function} of a truncated convex cusp $M_{tr} \in \M_{tr}(\Tor,g)$ is a map $(\Tor, g) \to \R$ that associates to every point on the hyperbolic boundary of $M_{tr}$ its distance from the Euclidean boundary.
\end{definition}
Due to Lemma \ref{lemmadistancelike} we know that in the interior of every face $F$ of $M_{tr}$ the distance function has the form
\begin{equation} \label{eqn:DistLike}
x \mapsto \log \cosh (\dist(x,a)) + b,
\end{equation}
where  $b > 0$ and $a$ is a point in $F$ or in the hyperbolic plane spanned by $F$.

We call a function of the form (\ref{eqn:DistLike}) on a subset of the hyperbolic plane a \emph{distance-like function}.

\begin{definition}
A function $f: (\Tor, g) \to \R$ is called \emph{piecewise distance-like function}, briefly \emph{PD function}, if there exists a geodesic triangulation $T$ of $(\Tor, g)$ such that $f$ is distance-like on every triangle of $T$.

A PD function $f$ is called \emph{$Q$-concave} if for every geodesic arc $\gamma$ on $(\Tor, g)$ at every kink point of the restriction $f|_\gamma$ the left derivative is greater than the right derivative.
\end{definition}
Figure \ref{fig:PDFunction} shows an example of a Q-concave PD function on the line.

Recall that a triangle $ijk$ of $T$ may have identifications on the boundary; so we mean by a distance-like function on $ijk$ a function induced from a distance-like function on its development.

\begin{figure}[ht]
\centerline{\input{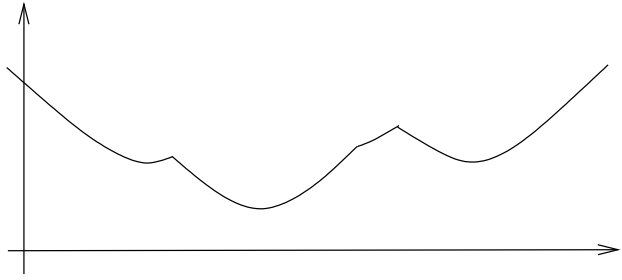}}
\caption{The graph of a Q-concave PD function.}
\label{fig:PDFunction}
\end{figure}

The following lemma is straightforward.
\begin{lemma}
The distance function of a truncated convex cusp is a Q-concave PD function. Conversely, every positive Q-concave PD function is the distance function of a unique truncated convex cusp.
\end{lemma}
Thus we can identify the space $\M_{tr}(\Tor,g)$ with the space of positive Q-concave PD functions on $(\Tor,g)$.

\begin{definition}
Let $T$ be a geodesic triangulation of $(\Tor, g)$, and let $h: \Sigma \to \R$ be a function on the singular set of $(\Tor, g)$. By $\widetilde{h_T}: (\Tor, g) \to \R$ we denote the PD extension of the function $h$ with respect to the triangulation $T$.
\end{definition}
The function $\widetilde{h_T}$ does not always exist, but it is easy to see that if it does, then it is well-defined and unique.

Extending a function $h: \Sigma \to \R$ to a positive Q-concave PD function is equivalent to constructing a truncated convex cusp with truncated particle lengths $h$. Therefore Proposition \ref{prp:HeightsToCusp} will follow from
\begin{lemma}
Let $h$ be a function on $\Sigma$ and let $T, T'$ be two geodesic triangulations of $(\Tor,g)$. If both functions $\widetilde{h_T}$ and $\widetilde{h_{T'}}$ exist and are $Q$-concave, then they are equal.
\end{lemma}
\begin{proof}
Let $x$ be an intersection point of an edge $e$ of $T$ and an edge $e'$ of $T'$. Then the function $\widetilde{h_T}$ is distance-like on $e$, and the function $\widetilde{h_{T'}}$ is PD Q-concave on $e$. By Lemma \ref{lem:BoundsForPD}, this implies $\widetilde{h_{T'}}(x) \ge \widetilde{h_T}(x)$. Considering the edge $e'$ instead of $e$, we get an inverse inequality. Hence $\widetilde{h_T}(x) = \widetilde{h_{T'}}(x)$. The union of edges of $T$ and $T'$ subdivides $(\Tor, g)$ into hyperbolic polygons such that both functions $\widetilde{h_T}$ and $\widetilde{h_{T'}}$ are distance-like on every polygon of the subdivision. As we just proved, $\widetilde{h_T}(x)$ and $\widetilde{h_{T'}}(x)$ take equal values at the vertices of the polygons. It follows that $\widetilde{h_T} = \widetilde{h_{T'}}$.
\end{proof}

Proposition \ref{prp:HeightsToCusp} is proved. It implies that the map
$$
\begin{array}{rcl}
\M_{tr}(\Tor, g) & \to & \R^\Sigma\\
(T, h) & \mapsto & h
\end{array}
$$
is an embedding. Changing a truncation of a convex cusp results in adding a common constant to all of the truncated particle lengths. Thus we have an embedding
$$
\begin{array}{rcl}
\M(\Tor, g) & \to & \R^\Sigma / \langle \mathbb{1} \rangle\\
(T, [h]) & \mapsto & [h],
\end{array}
$$
where $[h]$ is an equivalence class under $(h_1, \ldots, h_n) \sim (h_1 + c, \ldots, h_n + c)$.

For a geodesic triangulation $T$ of $(\Tor, g)$, denote by $\M_{tr}^T(\Tor, g)$ the space of truncated convex cusps that have a representative of the form $(T, h)$. In other words, $M_{tr} \in \M_{tr}^T(\Tor, g)$ iff $M_{tr}$ can be cut into horoprisms over the triangulation $T$. We have a decomposition
$$
\M_{tr}(\Tor, g) = \bigcup_T \M_{tr}^T(\Tor, g).
$$
Clearly, we have a similar decomposition for $\M(\Tor,g)$, where
$$
\M^T(\Tor,g) = \M_{tr}^T(\Tor, g)/\langle \mathbb{1} \rangle.
$$

Let us denote by $\widetilde{h}$ the distance function of the truncated convex cusp with truncated particle lengths $h = (h_i)_{i\in \Sigma}$. In other words, $\widetilde{h}$ is the unique Q-concave PD extension of the function $h: i \mapsto h_i$.

\subsection{Lemmas} \label{subsec:Lemmas} We put here lemmas used in the proof of Proposition \ref{subsec:MTg}.

\begin{lemma} \label{lem:DistFuncBound}
The distance function $\widetilde{h}$ of a truncated convex cusp satisfies the inequality
$$
|\widetilde{h}(x) - \widetilde{h}(y)| < \dist(x,y)
$$
for any $x,y \in \Tor$, where $\dist$ denotes the shortest path distance for the metric $g$. In particular,
$$
\max_\Tor \widetilde{h} - \min_\Tor \widetilde{h} < \diam(T,g).
$$

\end{lemma}
\begin{proof} This follows from the fact that the gradient of the distance function is always smaller than $1$. \end{proof}

\begin{lemma} \label{lem:TriangFin}
The space $\M_{tr}^T(\Tor, g)$ is non-empty only for finitely many geodesic triangulations $T$ of $(\Tor, g)$.
\end{lemma}
\begin{proof} The proof proceeds in two steps. First, we show that there is a constant $L$ depending on the metric $g$, such that no triangulation associated with a cusp with boundary $(\Tor,g)$ has an edge of length greater than $L$. Second, we note that there are only finitely many geodesic arcs of length $\le L$ between points of $\Sigma$. Then the number of geodesic triangulations with edges of length $\le L$ is also finite and we are done.

Let $\widetilde{h}$ be the distance function of a truncated convex cusp, and let $e$ be an edge of an associated triangulation. The restriction of $\widetilde{h}$ to $e$ is a function of the form $\log\cosh(x-a) + b$,  where $x$ is the arc length parameter on $e$. It can easily be shown that for any $C \in \R$ there exists an $L \in \R$ such that
$$
\max_e \widetilde{h} - \min_e \widetilde{h} > C,
$$
as soon as the length of $e$ is greater than $L$. Put $C = \mathrm{diam\,}(\Tor, g)$. Then the length of $e$ cannot be greater than $L$ due to Lemma \ref{lem:DistFuncBound}.

The lengths of geodesic arcs between the points of $\Sigma$ form a discrete subset of $\R$ by the argument from \cite[Proposition 1]{ILTC}.
\end{proof}

\textit{Remark.} Lemmas \ref{lem:DistFuncBound} and \ref{lem:TriangFin} hold also for non-convex polyhedral cusps with particles.

\begin{definition} \label{dfn:Tilde}
For two points $i,j$ on the real line and two real numbers $h_i, h_j$, let
$$
\widetilde{h_{ij}}: \R \to \R
$$
denote the distance-like function that takes values $h_i, h_j$ at $i$ and $j$, respectively.
\end{definition}
Note that the function $\widetilde{h_{ij}}$ exists iff $|h_i-h_j| < \dist(i,j)$.

\begin{lemma} \label{lem:BoundsForPD}
For any Q-concave PD function $\widetilde{h}$ on $\R$ such that $\widetilde{h}(i) = h_i, \widetilde{h}(j) = h_j$ the following holds:
$$
\begin{array}{cc}
\widetilde{h}(x) \ge \widetilde{h_{ij}}(x) & \mbox{for all } x \in [i,j],\\
\widetilde{h}(x) \le \widetilde{h_{ij}}(x) & \mbox{for all } x \notin [i,j].
\end{array}
$$
\end{lemma}
\begin{proof}
Consider the function $\widetilde{h} - \widetilde{h_{ij}}$. There are numbers $x_1 < x_2 < \cdots < x_n$ and $a_0, b_0, \ldots, a_n, b_n$ such that
$$
\widetilde{h}|_{[x_m, x_{m+1}]} = \log\cosh(x-a_m) + b_m
$$
for all $m$ from $0$ to $n$, where we put $x_0 = -\infty, x_{n+1} = +\infty$. Since $\log\cosh x$ is a convex function, $(\widetilde{h} - \widetilde{h_{ij}})|_{[x_m, x_{m+1}]}$ is a monotone function for every $m$. Besides, since $\widetilde{h}$ is Q-concave, $\widetilde{h} - \widetilde{h_{ij}}$ is Q-concave too. Thus, if $\widetilde{h} - \widetilde{h_{ij}}$ is monotone decreasing (or constant) on $[x_{m-1}, x_m]$, then it is also monotone decreasing (or constant) on $[x_m, x_{m+1}]$. Together with
$$
(\widetilde{h} - \widetilde{h_{ij}})(i) = (\widetilde{h} - \widetilde{h_{ij}})(j) = 0
$$
this implies that the function $\widetilde{h} - \widetilde{h_{ij}}$ is non-negative on the interval $[i,j]$ and non-positive outside of it. The lemma follows.
\end{proof}

\begin{definition} \label{dfn:Slope}
Let $\widetilde{h}(x) = \log \cosh (\dist(x,a)) + b$ be a distance-like function on a subset of $\H^2$. We call the distance $\dist(x,a)$ the \emph{slope} of $\widetilde{h}$ at $x$ and denote it by $\mathrm{slope}_x(\widetilde{h})$.
\end{definition}
For a PD function $\widetilde{h}$ on the torus, $\mathrm{slope}_x(\widetilde{h})$ is defined in an obvious way, provided that $\widetilde{h}$ is locally distance-like at $x$. Clearly, $\mathrm{slope}_x(\widetilde{h})$ depends only on the
gradient norm of $\widetilde{h}$ at $x$, and it tends to $\infty$ as the gradient norm tends to $1$.

\begin{lemma} \label{lem:SlopeBound}
The slopes of the convex cusps with the boundary $(\Tor,g)$ are uniformly bounded. That is, there exists a constant $D \in \R$ such that
$$
\mathrm{slope}_x(\widetilde{h}) \le D
$$
for every Q-concave PD function $\widetilde{h}$ on $(\Tor, g)$ at every point $x \in \Tor$.
\end{lemma}
\begin{proof}
Let us show that the lemma holds for $\log \cosh D = \diam(\Tor, g)$. Assume the converse, and let $\widetilde{h}$ and $x$ be such that $\mathrm{slope}_x(\widetilde{h}) > D$. On the geodesic that starts at $x$ and runs in the direction of $-\mathrm{grad}_x(\widetilde{h})$, take the point $y$ at a distance $D$ from $x$. If we end up at a singular point before running the distance $D$, then perturb the point $x$ so that $\mathrm{slope}_x(\widetilde{h})$ is still greater than $D$. On the geodesic arc $xy$, consider functions $\widetilde{h}$ and $\widetilde{h_{x+}}$, where $\widetilde{h_{x+}}$ is the distance-like function that coincides with $\widetilde{h}$ in a neighborhood of $x$. By Lemma \ref{lem:BoundsForPD}, we have
$$
\widetilde{h}(y) \le \widetilde{h_{x+}}(y).
$$
Due to $\mathrm{slope}_x(\widetilde{h}) > |xy|$, the function $\widetilde{h_{x+}}$ is monotone decreasing on $xy$. The convexity of $\log \cosh x$ implies
$$
\widetilde{h}(x) - \widetilde{h_{x+}}(y) \ge  \log \cosh |xy| = \log \cosh D = \diam(\Tor, g).
$$
Therefore $\widetilde{h}(x) - \widetilde{h}(y) \ge \diam(\Tor, g) \ge \dist(x,y)$ which contradicts Lemma \ref{lem:DistFuncBound}.
\end{proof}

Let us generalize Definition \ref{dfn:Tilde} to the situation when $i,j \in \Sigma$ are singular points of $(\Tor,g)$ joined by a geodesic arc $\gamma$. Then $\widetilde{h_{ij}}$ is the distance-like function on $\gamma$ that takes values $h_i$ and $h_j$ at $i$ and $j$, respectively. Note that there are many arcs that join $i$ and $j$, so we need to specify $\gamma$ when we talk about $\widetilde{h_{ij}}$. If $h \in \M_{tr}(\Tor,g)$ and $\gamma$ is an edge of a triangulation associated with $h$, then $\widetilde{h_{ij}} = \widetilde{h}|_\gamma$. Also we might want to extend function $\widetilde{h_{ij}}$ beyond the point $i$. For this we consider a \emph{geodesic extension} of $\gamma$ beyond $i$. This is a geodesic ray from $i$ that forms the angle $\pi$ with $ij$. We measure the angles around $i$ modulo the cone angle $\alpha_i$ at $i$, so the geodesic extension is defined for $\alpha_i< \pi$ as well. 
In general, there are two geodesic extensions (``to the left'' and ``to the right''), and they coincide only if $\alpha_i = \frac{2\pi}{n}$ for some $n$.

\begin{lemma} \label{lem:hikhjl}
Let $i,j,k,l \in \Sigma$. Choose geodesic arcs $ik$ and $jl$ and a geodesic extension of $ik$ beyond $i$. Suppose that the extension of $ik$ intersects the arc $jl$ at a point $m$, see Figure \ref{fig:hikhjl}. Then for every $h \in \M_{tr}(\Tor, g)$ holds
\begin{equation} \label{eqn:hikhjl}
\widetilde{h_{ik}}(m) \ge \widetilde{h_{jl}}(m).
\end{equation}
In particular, let $i, j \in \Sigma$ be such that there is a closed geodesic arc based at $j$ that bounds a disk in $\Tor$ such that $i$ is the only singularity inside this disk. Then for every $h \in \M_{tr}(\Tor, g)$ holds
\begin{equation} \label{eqn:hihj}
h_i \ge h_j - \log \cosh \ell_{ij},
\end{equation}
where $\ell_{ij}$ is the length of the geodesic arc $ij$ that lies inside the disk. See Figure \ref{fig:hikhjl}.
\end{lemma}
\begin{figure}[ht]
\centerline{\input{hikhjl.pstex_t}}
\caption{To Lemma \ref{lem:hikhjl}.}
\label{fig:hikhjl}
\end{figure}
\begin{proof}
Consider the restriction of the Q-concave PD function $\widetilde{h}$ to the piecewise geodesic arc $kim$. This is a PD function, and since the arc $kim$ can be approximated by geodesic arcs, it is also Q-concave. Then by Lemma \ref{lem:BoundsForPD} we have
$$
\widetilde{h}(m) \le \widetilde{h_{ik}}(m).
$$
By Lemma \ref{lem:BoundsForPD} applied to the arc $jl$, we have
$$
\widetilde{h}(m) \ge \widetilde{h_{jl}}(m).
$$
Inequality (\ref{eqn:hikhjl}) follows.

Let us derive (\ref{eqn:hihj}) from (\ref{eqn:hikhjl}). Consider the function
$$
f(x) = \log \cosh (\dist(x,i)) + h_i
$$
on the singular disk bounded by the arc $jj$. We have $h_i = f(i)$. If also $h_j = f(j)$, then we have
$$
\widetilde{h_{ij}}(m) = f(m) = \widetilde{h_{jj}}(m),
$$
where $m$ is the first intersection point of the geodesic extension of $ij$ beyond $i$ with the arc $jj$. It is easy to see that $\widetilde{h_{ij}}(m)$ is a monotone decreasing, and $\widetilde{h_{jj}}(m)$ is a monotone increasing function of $h_j$. Thus, since we have $\widetilde{h_{ij}}(m) \ge \widetilde{h_{jj}}(m)$ due to (\ref{eqn:hikhjl}), we must have $h_j \le f(j)$, and this is exactly the inequality (\ref{eqn:hihj}).
\end{proof}

\begin{lemma} \label{lem:Conv}
Let $i,j,k$ be three points on the real line such that $j$ lies between $i$ and $k$. Then the function $\widetilde{h_{ij}}(k)$ is a concave function of $h_i$, $h_j$, and the function $\widetilde{h_{ik}}(j)$ is a convex function of $h_i$, $h_k$.
\end{lemma}
\begin{proof}
By Lemma \ref{lem:SinhForm},
$$
\widetilde{h_{ik}}(j) = \log (a e^{h_i} + b e^{h_k})
$$
with positive $a$ and $b$. The Hessian can easily be computed and shown to be positive semidefinite. For $\widetilde{h_{ij}}(k)$ one has a similar expression with one positive and one negative coefficient.
\end{proof}

\subsection{Description of $\M(\Tor, g)$} \label{subsec:MTg}
Here we prove the main result of this section. Recall that the space $\M_{tr}(\Tor, g)$ is identified with a subset of $\R^\Sigma$ by associating to a truncated cusp the truncated particle lengths $h = (h_i)_{i \in \Sigma}$. The space $\M(\Tor, g)$ is thus identified with a subset of $\R^\Sigma / \la \mathbb{1} \ra$.
\begin{proposition} \label{prp:SpaceOfCusps}
The space $\M(\Tor, g)$ of convex polyhedral cusps with particles with boundary $(\Tor,g)$ has the following properties:
\begin{enumerate}
\item It is a non-empty compact convex subset of $\R^\Sigma / \la \mathbb{1} \ra$ with non-empty interior.
\item If all faces of a convex cusp $M \in \M(\Tor, g)$ are strictly convex hyperbolic polygons (after developing on $\H^2$), then $M$ is an interior point of $\M(\Tor, g)$.
\item For a geodesic triangulation $T$ of $(\Tor,g)$, let $\M^T(\Tor,g) \subset \M(\Tor,g)$ be the space of cusps whose faces are unions of triangles from $T$. Then the decomposition
\begin{equation} \label{eqn:Chambers}
\M(\Tor, g) = \bigcup_T \M^T(\Tor, g)
\end{equation}
is finite, and every $\M^T(\Tor, g)$ is a compact set with piecewise analytic boundary.
\end{enumerate}
\end{proposition}
\begin{proof} $\M(\Tor,g) \ne \emptyset:$ We claim that $[0,\ldots,0] \in \M(\Tor, g)$, that is there exists an isosceles cusp with particles with boundary $(\Tor,g)$. Furthermore, faces of this cusp are the faces of the Delaunay tesselation of $(\Tor,g)$, where the Delaunay tesselation of a surface with a cone metric is defined as the dual of the Voronoi tesselation, see \cite[Proposition 3.1]{Thurart1}. To show this, let $T_D$ be a Delaunay triangulation, that is a refinement of the Delaunay tesselation of $(\Tor,g)$. Inscribe a triangle $ijk$ of $T_D$ in a horosphere with center $c$. This gives an isosceles semi-ideal pyramid $cijk$. Let $ijl$ be a triangle of $T_D$ adjacent to $ijk$. Develop $ijl$ into the hyperbolic plane spanned by $ijk$. By the main property of Delaunay tesselations, the point $l$ lies outside or on the circumcircle of $ijk$. Hence, $l$ lies outside or on the horosphere through $i,j$ and $k$. It follows that the truncated length of $cl$ is larger than the truncated length of $ci, cj$ and $ck$. In order to make the pyramid $cijl$ isosceles, one has to rotate the triangle $ijl$ around the edge $ij$ towards $c$. As a result, the total dihedral angle of the pyramids $cijk$ and $cijl$ at the edge $ij$ becomes $\le \pi$.

Property (3): The finiteness of the decomposition (\ref{eqn:Chambers}) is proved in Lemma \ref{lem:TriangFin}. A point $[h] \in \R^\Sigma / \la \mathbb{1} \ra$ lies on the boundary of $\M^T(\Tor, g)$ iff the function $\widetilde{h}$ is distance-like across some of the edges of $T$. Clearly, for every edge this condition is analytic. It remains to show that $\M^T(\Tor,g)$ is compact. Embed a triangle $ijk \in \T$ into $\H^3$. The hyperbolic plane spanned by $ijk$ divides the sphere at infinity in two open hemispheres. The space of semi-ideal pyramids with the base $ijk$ is naturally homeomorphic to any one of these hemispheres. For a point $c \in \partial \overline{\H^3}$ not coplanar with $ijk$, call the \emph{slope} ot the pyramid $cijk$ the maximum distance between the projection of $c$ on the plane spanned by $ijk$ and a point of the triangle $ijk$. Clearly, the slope of $cijk$ is the maximum slope of the corresponding distance function on $ijk$, see Definition \ref{dfn:Slope}. For any $D \in \R$, the space of semi-ideal pyramids with slope $\le D$ is compact. By Lemma \ref{lem:SlopeBound}, there exists $D$ such that all pyramids in convex cusps with boundary $(\Tor, g)$ have slopes $\le D$. The conditions $\theta_{ij} \le \pi$ for all $ij \in T$ are closed ones. Thus $\M^T(\Tor,g)$ is compact as a closed subset of a compact space.

Property (2): Let $(T,h)$ be a pair associated with a truncation of $M$. There are two types of edges of $T$: \emph{true edges} that are edges of the cusp $M$ and \emph{flat edges} that were added to refine the face decomposition to a triangulation. 
We want to show that there exists an $\epsilon > 0$ such that $h' \in \M_{tr}(\Tor, g)$ for all $h'$ in $\R^\Sigma$ at a distance $<\epsilon$ from $h$. That is, for every such $h'$ we want to find a triangulation $T'$ such that the function $\widetilde{h'_{T'}}$ exists and is Q-concave.

We obtain a triangulation $T'$ from the triangulation $T$ by the \emph{flip algorithm}. Let $ij$ be an edge of $T$ such that $\widetilde{h'_T}$ is not Q-concave across $ij$. We call such an edge a \emph{bad edge} of $T$. If $ij$ belongs to two different triangles $ijk$ and $ijl$ of $T$ and the quadrilateral $ikjl$ is strictly convex, then the edge $ij$ can be \emph{flipped}, that is replaced by the diagonal $kl$. The flip algorithm produces a sequence of triangulations $T^0 = T, T^1, T^2,\ldots$, where $T^{n+1}$ is obtained from $T^n$ by flipping a bad edge of $T^n$. If the algorithm terminates at a triangulation $T'$ that has no bad edges, then we are done. But some things might go wrong. First, the PD function $\widetilde{h'_{T^n}}$ might not exist for some $n$. Second, it might be impossible to flip a bad edge, see Figure \ref{fig:NoFlip}. And third, the algorithm might run infinitely. Let us show that none of these occurs in our particular case.

\begin{figure}[ht]
\centerline{\input{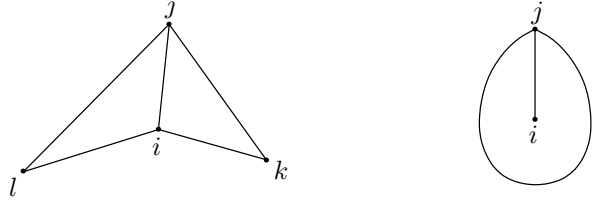}}
\caption{The two situations where the edge $ij$ cannot be flipped.}
\label{fig:NoFlip}
\end{figure}

Clearly, the function $\widetilde{h'_T}$ exists, if $\epsilon$ is sufficiently small. It is also easy to see that, for small $\epsilon$, a true edge of $T$ can never become bad and thus will never be flipped. Thus the flip algorithm is performed independently inside every face. This implies that every triangle of $T^n$ is contained in a face of $M$, and thus the function $\widetilde{h'_{T^n}}$ exists, for small $\epsilon$. Since every face is a convex polygon, situations on Figure \ref{fig:NoFlip} cannot occur, and a bad edge can always be flipped. The algorithm is finite since the function $\widetilde{h'_{T^{n+1}}}$ is pointwise greater or equal than $\widetilde{h'_{T^n}}$ and since every face has finitely many triangulations.

Property (1): The space $\M(\Tor,g)$ is compact since it is the union of finitely many compact spaces, by Property (3). The Delaunay cusp constructed at the beginning of the proof has convex faces and thus is an interior point of $\M(\Tor,g)$ due to Property (2). It remains to prove the convexity of $\M(\Tor,g)$. We will show that $\M_{tr}(\Tor, g)$ is convex, that is for every $h^0, h^1 \in \M_{tr}(\Tor, g)$ and every $0<\lambda<1$ the point
$$
h^\lambda := (1-\lambda)h^0 + \lambda h^1
$$
also belongs to $\M_{tr}(\Tor, g)$. Then $\M(\Tor,g)$ is convex as a projection of $\M_{tr}(\Tor,g)$.

Since $\M(\Tor, g)$ is closed, it suffices to prove $h^\lambda \in \M_{tr}(\Tor,g)$ for all sufficiently small $\lambda$. Our proof of it uses the flip algorithm described in the proof of Property~(2).

Let $T$ be a triangulation associated with $h^0$. If $\lambda$ is sufficiently small, then the function $\widetilde{h^\lambda_T}$ exists. The true edges of $T$ will never be flipped, thus the flip algorithm is performed independently inside every face of $\widetilde{h^0}$. It follows that for every triangulation $T^n$ that appears during the algorithm, the function $\widetilde{h^\lambda_{T^n}}$ exists. Every face can be triangulated in only finitely many ways, since otherwise infinitely many triangulations would be associated with $h^0$, which contradicts Lemma \ref{lem:TriangFin}. Hence the algorithm cannot run infinitely. It remains to show that a bad edge can always be flipped.

Assume that $ij$ is a bad edge in the situation on the left of Figure \ref{fig:NoFlip}. Badness means
\begin{equation} \label{eqn:BadEdge}
\widetilde{h^\lambda_{ik}}(m) < \widetilde{h^\lambda_{jl}}(m),
\end{equation}
where $m$ is the intersection point of the lines $ik$ and $jl$. On the other hand, by Lemma \ref{lem:hikhjl} for both $h^0$ and $h^1$ the opposite inequality (\ref{eqn:hikhjl}) holds. By Lemma \ref{lem:Conv}, $\widetilde{h_{ik}}(m)$ is a concave and $\widetilde{h_{jl}}(m)$ is a convex function on $h_i$ and $h_k$. This implies that the inequality (\ref{eqn:BadEdge}) is false and $ij$ is not bad. If $ij$ cannot be flipped because of the situation on the right of Figure \ref{fig:NoFlip}, then the same argument applies since the inequality (\ref{eqn:hihj}) is a convex condition on $h_i$ and $h_j$.
\end{proof}

\emph{Remarks.} A similar analysis of the space of ``generalized'' convex polytopes and ``generalized'' convex caps was made in \cite{BobenkoIzmestiev} and \cite{Izmestiev}. In the former case the distance function was the (square of) the distance from a point to a plane in $\R^3$, in the latter case --- the distance from a variable point in a plane to another plane in $\R^3$. Thus, instead of $\log\cosh$, the distance-like functions in \cite{BobenkoIzmestiev} and \cite{Izmestiev} were modelled on $x^2$ and on $ax$ with a real parameter $a \in [-1,1]$, respectively. In both cases one succeeded to describe the space of ``generalized'' objects explicitely by linear and quadratic inequalities on coordinates. The arguments from \cite{BobenkoIzmestiev} and \cite{Izmestiev} cannot be carried out in the present paper because the function $\widetilde{h_{ij}}$ from Definition \ref{dfn:Tilde} does not always exist.

It is now easy to prove Theorem \ref{Base theorem} in the special case when the metric $g$ has only one singularity. Indeed, the isosceles cusp over the Delaunay tesselation of $(\Tor,g)$ is convex, see the first paragraph of the proof of Proposition \ref{prp:SpaceOfCusps}. The curvature $\kappa$ of its only particle vanishes due to $\sum_i \kappa_i = 0$. Thus it is a convex polyhedral cusp. This cusp is unique, since the space $\M(\Tor,g)$ consists of a single point. In the case of one singularity the results of the subsequent sections have no real meaning. If there are two singularities, then the space $\M(\Tor,g)$ is a segment in $\R$, there can be different triangulations $T$ such that $\M^T(\Tor,g)$ is non-empty, and the things become more interesting.

\section{The total scalar curvature}\label{sec:scalar curvature}
In this section we define the total scalar curvature function $S$ on the space $\M(\Tor,g)$, compute its derivatives and show that $S$ is strictly concave.

\subsection{Derivatives of the total scalar curvature}
\begin{definition}
Let $M \in \M(\Tor,g)$ be a convex polyhedral cusp with particles. The \emph{total scalar curvature} of $M$ is defined as
$$
S(M) = -2\mathrm{Vol}(M) + \sum_{i \in \Sigma} h_i \kappa_i + \sum_{e \in \mathcal{E}(M)} \ell_e(\pi - \theta_e).
$$
Here $(h_i)_{i\in \Sigma}$ are the truncated particle lengths for an arbitrary truncation of~$M$, $\ell_e$ is the length of the edge $e$, and $\theta_e$ is the dihedral angle of $M$ at $e$. The second sum ranges over the set $\E(M)$  of edges of $M$. 
\end{definition}
Choosing a different truncation of $M$ does not change the values of $\kappa_i, \ell_e$ and $\theta_e$. The truncated particle lengths $h_i$ are all changed by the same amount. Thus, by Lemma \ref{Lem:sumcurvaturearezero}, the value $S(M)$ is well-defined.

\begin{lemma}
The function $S$ is twice continuously differentiable on $\M(\Tor,g)$. Its first partial derivatives are:
\begin{equation} \label{eqn:dSdh}
\frac{\partial S}{\partial h_i}  =  \kappa_i.
\end{equation}
The second partial derivatives of $S$ have the properties:
\begin{eqnarray}
\frac{\partial^2 S}{\partial h_i \partial h_j}\!\!\!\! && \!\!\left\{ \begin{array}{ll}
							>0 & \mbox{if }ij \in \E(M),\\ \label{eqn:dSdhij}
							=0 & \mbox{otherwise,}
                                                       \end{array} \right.\\ \label{eqn:dSdhii}
\frac{\partial^2 S}{\partial h_i^2} & = & - \sum_{j\ne i} \frac{\partial^2 S}{\partial h_i \partial h_j}.
\end{eqnarray}
\end{lemma}
Recall that $\M(\Tor, g) \subset \R^\Sigma / \la \mathbb{1} \ra$. So, by $\frac{\partial}{\partial h_i}$ we mean the directional derivative in the direction of $[e_i] = e_i + \la \mathbb{1} \ra$, where $e_i$ is the $i$th basis vector of $\R^\Sigma$.
\begin{proof}
Let $T$ be a triangulation of $(\Tor,g)$ such that $M \in \M^T(\Tor,g)$. Cut $M$ into semi-ideal triangular pyramids according to the triangulation $T$. If $M$ is an interior point of $\M^T(\Tor,g)$, then the triangulation $T$ is associated to all cusps near $M$. By a generalization of the Schl\"afli formula, \cite[page 294]{Mil94}, \cite[Theorem 14.5]{Riv94}, we have
$$
d\Vol(M) = \frac{1}{2} \sum_i h_i d\kappa_i -\frac{1}{2} \sum_e \ell_e d\theta_e.
$$
This implies (\ref{eqn:dSdh}) in the case when $M$ is an interior point of some $\M^T(\Tor,g)$. In the general case, let $\xi \in \R^\Sigma$ be such that $[h+t\xi] \in \M(\Tor,g)$ for all sufficiently small $t$. Due to the piecewise analyticity of the boundaries of $\M^T(\Tor, g)$, there exists a triangulation $T$ such that $[h+t\xi] \in \M^T(\Tor,g)$ for all sufficiently small $t$. Therefore the previous argument shows that
$$
\frac{\partial S}{\partial \xi} = \sum_{i \in \Sigma} \xi_i \kappa_i.
$$

The same argument can be applied to show that $S \in C^2(\M(\Tor,g))$. Thus we can concentrate on computing the derivatives $\frac{\partial \kappa_i}{\partial h_j}$ for $M$ in the interior of some $\M^T(\Tor,g)$. This is reduced to computing the derivatives in a single semi-ideal pyramid. Introduce notations for the angles as on Figure \ref{prisms}.

\begin{figure}[ht] \begin{center}
\input{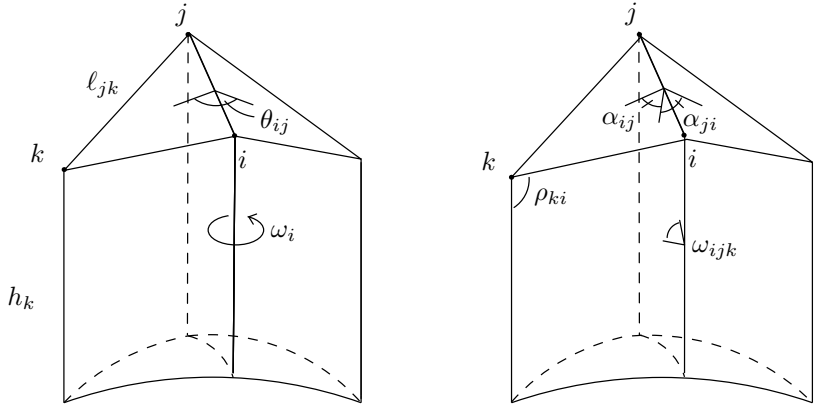}
\end{center}
\caption{Angles and lengths in two adjacent horoprisms (truncated semi-ideal pyramids). The curved triangles are the Euclidean bases of the horoprisms. \label{prisms}}
\end{figure}

The angle $\omega_{ijk}$ can be viewed as a function of the angles $\rho_{ij}$ and $\rho_{ik}$. Thus we have
\begin{equation} \label{eqn:dodh}
\frac{\partial \omega_{ijk}}{\partial h_j} =  \frac{\partial \omega_{ijk}}{\partial \rho_{ij}} \frac{\partial \rho_{ij}}{\partial h_j}.
\end{equation}
From (\ref{eqn:cos}) we compute
$$
\frac{\partial \rho_{ij}}{\partial h_j} = \frac{e^{h_j-h_i}}{\sinh\ell_{ij} \sin\rho_{ij}}.
$$
From the link of the vertex $i$ in the pyramid over $ijk$ with the help of spherical sine and cosine laws we compute
$$
\frac{\partial \omega_{ijk}}{\partial \rho_{ij}} = - \frac{\cot\alpha_{ij}}{\sin\rho_{ij}}.
$$
Substituting in (\ref{eqn:dodh}), we obtain
$$
\frac{\partial \omega_{ijk}}{\partial h_j} = - \frac{e^{h_j-h_i} \cot\alpha_{ij}}{\sinh\ell_{ij} \sin^2\rho_{ij}}
$$
and
\begin{equation} \label{eqn:dkdh}
\frac{\partial\kappa_i}{\partial h_j} = - \frac{\partial \omega_i}{\partial h_j} = \frac{e^{h_j-h_i} (\cot\alpha_{ij} + \cot\alpha_{ji})}{\sinh\ell_{ij} \sin^2\rho_{ij}}.
\end{equation}
Note that $\cot\alpha_{ij} + \cot\alpha_{ji} > 0$ if $\alpha_{ij} + \alpha_{ji} < \pi$, and vanishes if $\alpha_{ij} + \alpha_{ji} = \pi$. Since $\alpha_{ij} + \alpha_{ji} = \theta_{ij}$, this implies (\ref{eqn:dSdhij}).

The equation (\ref{eqn:dSdhii}) is equivalent to
$$
\sum_j \frac{\partial \kappa_j}{\partial h_i} = 0,
$$
which is true due to $\sum_j \kappa_j = 0$.
\end{proof}

\noindent\emph{Remarks}. Note that the equation (\ref{eqn:dkdh}) holds only if there is a unique edge between the vertices $i$ and $j$. In the case of multiple edges, one has to sum up the right hand side of (\ref{eqn:dkdh}) over all edges between $i$ and $j$.

In general, $S$ is not of class $C^3$. For example, if $\theta_{ij} = \pi$ and $\theta_{ik} \not= \pi$ on Figure~\ref{prisms}, then
$$
\frac{\partial}{\partial h_k} \frac{\partial\kappa_i}{\partial h_j} = 0, \mbox{ but } \frac{\partial}{\partial h_j} \frac{\partial\kappa_i}{\partial h_k} \ne 0.
$$

\subsection{Concavity of the total scalar curvature}
Let $\Gamma(M)$ be an embedded graph in $(\Tor,g)$ with vertex set $\Sigma$ and edge set $\E(M)$ that consists of the edges of the convex polyhedral cusp with particles $M$.
\begin{lemma} \label{lem:HessKernel}
The Hessian of the function $S$ is negatively semidefinite. The nullspace of $(\frac{\partial^2 S}{\partial h_i \partial h_j})$ is spanned by the vectors $v^K = [h_i^K]_{i \in \Sigma }$ defined as
$$
h_i^K = \left\{ \begin{array}{ll}
		1, & \mbox{for }i \in K;\\
		0, & \mbox{for }i \notin K,
                \end{array}
\right.
$$
where $K \subset \Sigma $ is a connected component of $\Gamma(M)$.
\end{lemma}
\begin{proof}
Denote $a_{ij} = \frac{\partial^2 S}{\partial h_i \partial h_j}$. Due to (\ref{eqn:dSdhii}), we have
$$
\sum_{i,j}a_{ij}x_ix_j = -\sum_{i<j}a_{ij}(x_i-x_j)^2.
$$
Since $a_{ij} \ge 0$ by (\ref{eqn:dSdhij}), the Hessian is negatively semidefinite. By (\ref{eqn:dSdhij}) again, the Hessian vanishes on the vector $x$ if and only if $x_i$ is constant over every connected component of $\Gamma(M)$.
\end{proof}
As a consequence, $S$ is a concave function on $\M(\Tor,g)$. Points where the Hessian is degenerate can exist indeed, as the example in Subsection \ref{subsec:CuspsParticles} shows. However, in an important special case we can show that the Hessian is non-degenerate.

\begin{lemma}\label{lem:HessCusp}
If $M \in \M(\Tor,g)$ is a convex polyhedral cusp, then the Hessian of $S$ is not degenerate at $M$.
\end{lemma}
\begin{proof}
The developing map maps $M$ to a convex parabolic polyhedron, see Lemma \ref{lem:CuspToPol}. Since the 1-skeleton of a convex polyhedron is connected, the graph $\Gamma(M)$ is also connected. By Lemma \ref{lem:HessKernel}, the nullspace of the Hessian is spanned by the vector $v^\Sigma$ which projects to zero in $\R^\Sigma / \la \mathbb{1} \ra$.
\end{proof}
Note that convex polyhedral cusps correspond to critical points of $S$ due to (\ref{eqn:dSdh}).
\begin{corollary} \label{cor:Smax}
If $M \in \M(\Tor,g)$ is such that $\kappa_i(M) = 0$ for all $i \in \Sigma$, then
$$
S(M) > S(M')
$$
for every $M' \in \M(\Tor,g)$ different from $M$.
\end{corollary}

Although the Hessian of $S$ can degenerate at some points in $\M(\Tor, g)$, the following lemma holds.

\begin{lemma} \label{lem:SConc}
The function $S$ is strictly concave on $\M(\Tor,g)$, that is
$$
S[th + (1-t)h'] > tS[h] + (1-t)S[h']
$$
for all $[h] \ne [h']$ and all $t \in (0,1)$.
\end{lemma}
\begin{proof}
Assume this is not the case. Then there exist $[h] \ne [h']$ such that $S$ is linear on the segment joining $[h]$ and $[h']$. By choosing a subsegment, if necessary, we can assume $[h], [h'] \in \M^T(\Tor,g)$ for some triangulation $T$. This implies that the graph $\Gamma = \Gamma[h] \cup \Gamma[h']$ is embedded in $(\Tor, g)$.

Due to the linearity of $S$ between $[h]$ and $[h']$, the vector $[h'-h]$ belongs to the nullspace of the Hessian at both $[h]$ and $[h']$. Lemma \ref{lem:HessKernel} implies that the coordinate difference $h'_i - h_i$ is constant over every connected component of the graph $\Gamma$. Thus if $\Gamma$ is connected, then we have $[h] = [h']$, which is a contradiction.

Let $\Gamma$ be disconnected. Then its complement contains a non-simply connected component $F \subset (\Tor,g)$. Both functions $\widetilde{h}$ and $\widetilde{h'}$ are distance-like functions on $F$ (see Subsection \ref{subsec:HeightsToCusp}). By developing $F$ on a hyperbolic plane in $\H^3$, one easily sees that any two distance-like functions of $F$ differ by a constant. This implies that the coordinate difference $h'_i - h_i$ is constant over all vertices of $F$. Since this is true for every non-simply connected face of $\Gamma$, the difference $h'_i - h_i$ is constant over the whole $\Sigma$, and we have $[h] = [h']$, which is a contradiction.
\end{proof}

\section{Proofs of main theorems}\label{sec:proofs}

\subsection{Proof of Theorem \ref{Base theorem}.}
\emph{Existence}. Let $M$ be a convex polyhedral cusp with particles and with boundary $(\Tor,g)$. We will show that if some of its singular curvatures is not zero, then there exists a cusp with a larger total scalar curvature:
\begin{equation} \label{eqn:MaxScal}
\exists i \in \Sigma : \kappa_i(M) \ne 0 \quad \Longrightarrow \quad \exists M' \in \M(\Tor,g) : S(M') > S(M).
\end{equation}
Since the space $\M(\Tor,g)$ is non-empty and compact by Proposition \ref{prp:SpaceOfCusps}, the function $S$ attains its maximum at some point $M \in \M(\Tor,g)$. All of the curvatures of the cusp $M$ must vanish due to (\ref{eqn:MaxScal}). Thus (\ref{eqn:MaxScal}) implies the existence part of Theorem~\ref{Base theorem}.

The proof of (\ref{eqn:MaxScal}) is based on
\begin{lemma}[Volkov] \label{lem:Volkov}
Let $M$ be a convex polyhedral cusp with particles, and let $\kappa_i(M) < 0$ for an $i \in \Sigma$. Then every face containing the vertex $i$ has angle $<\pi$ at~$i$.
\end{lemma}
\begin{proof}
Consider the link of the vertex $i$ in $M$. This is a convex spherical polygon of perimeter $<2\pi$ with a cone singularity of negative curvature. The lemma says that its boundary cannot contain a geodesic of length $\ge \pi$. This is proved by Volkov in \cite{Vol55}, see also \cite[Lemma 9]{Izmestiev}.
\end{proof}

Let us prove (\ref{eqn:MaxScal}). Since $\sum_i \kappa_i = 0$, we may assume $\kappa_i < 0$. Due to (\ref{eqn:dSdh}), it suffices to show that the particle length $h_i$ can be decreased; in other words, that there exists a cusp $M' \in \M(\Tor,g)$ with truncated particle lengths
$$
\begin{array}{rcl}
h'_i & = & h_i - \epsilon, \\
h'_j & = & h_j \quad \mbox{for every } j \ne i,
\end{array}
$$
for a sufficiently small $\epsilon$. To prove the existence of $M'$, we have to find a polyhedral decomposition $\F'$ of $(\Tor,g)$ such that the horoprisms over the faces of $\F'$ with heights $h'$ exist and their total dihedral angles at the edges of $\F'$ are $\le \pi$.

Consider the face decomposition of $\partial M$. When we decrease $h_i$, nothing happens to the horoprisms over the faces that don't contain $i$. Let $F$ be a face of $M$ that contains $i$. By Lemma \ref{lem:Volkov}, the angle of $F$ at $i$ is $<\pi$. Let $j, k \in \Sigma$ be the vertices of $F$ adjacent to $i$. Draw in $F$ the shortest path $\gamma$ joining $j$ with $k$ and homotopic to the path $jik$, see Figure \ref{fig:Deformation}. Together $\gamma$ and $jik$ bound a polygon $P$ that has only three angles $< \pi$, namely those at $i,j$ and $k$. It is not hard to see that subdividing the polygon $P$ by diagonals from $i$ yields a desired decomposition of $(\Tor,g)$, Figure~\ref{fig:Deformation}.

\begin{figure}[ht] \begin{center}
\input{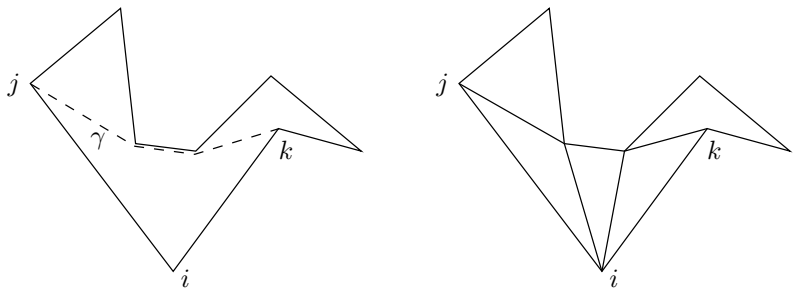}
\end{center}
\caption{Subdivision of the face $F$ when the height $h_i$ decreases. \label{fig:Deformation}}
\end{figure}

\emph{Uniqueness}. By Corollary \ref{cor:Smax}, if $M$ and $M'$ are two convex polyhedral cusps with boundary $(\Tor,g)$, then $S(M) > S(M')$ and $S(M') > S(M)$, which is a contradiction.

\subsection{Proof of Theorem \ref{cusps w particles determined by curvatures}.}
Theorem \ref{cusps w particles determined by curvatures} follows from Lemma \ref{lem:SConc} and the following proposition, see \cite[Lemma 6.1]{Luo06}, \cite[Proposition 5]{Izmestiev}.
\begin{lemma} \label{prp:ConvHomeo}
Let $f \in C^1(X)$ be a strictly convex or strictly concave function on a compact convex subset $X$ of a vector space $V$. Then the map $\grad f: X \to V^*$ is a homeomorphism onto the image.
\end{lemma}
Consider the function $S$ on $\M(\Tor,g)$. Since $S$ is strictly concave and $\grad S = \kappa$, Theorem \ref{cusps w particles determined by curvatures} follows.

Note that cusps with particles are in general not infinitesimally rigid.

\subsection{Proofs of Theorems \ref{InfRigCusp}, \ref{infinitesimal rigidity} and \ref{thm:weak infinitesimal rigidity}.}\label{sub:proofs rigidity}
The proofs are based on the fact that the Hessian of $S$ is equal to the Jacobian of the map $[h] \mapsto \kappa$:
$$
\frac{\partial^2 S}{\partial h_i \partial h_j} = \frac{\partial \kappa_i}{\partial h_j},
$$
see (\ref{eqn:dSdh}). Lemma \ref{lem:HessCusp} implies that for a cusp with zero curvatures the Jacobian has full rank, that is any non-trivial first-order variation of the particle lengths $[h]$ induces a non-trivial first-order variation of the curvatures $\kappa$.

Recall that a \emph{Killing field} of hyperbolic space  is  a vector
field of $ \mathbb{H}^3$ such that the elements of its local $1$-parameter group are
isometries. An \emph{infinitesimal isometric deformation} of a
polyhedral surface $S$ is a  Killing field on each face of a
triangulation of $S$ such that  two Killing fields on two adjacent triangles are equal on the common edge. The triangulation is required to have the same set of vertices as $S$. 
An infinitesimal isometric deformation is determined by its values at the vertices of $S$. It is called \emph{trivial}
if it is the restriction to $S$ of a global Killing
field. If all the infinitesimal isometric deformations of $S$
 are trivial, then $S$ is said to be \emph{infinitesimally rigid}.

Let $(P,G)$ be a convex parabolic polyhedron. An  infinitesimal isometric deformation $Z$ of $\partial P$ is called a \emph{parabolic deformation} if
\begin{equation}\label{eq:equivariant field}
Z(g (i))=dg.Z(i)+\vec{g}(i),
\end{equation}
where $i$ is a vertex of $P$ and $g\in G$. The vector $\vec{g}$ is a \emph{parabolic Killing field}. It is obtained as follows. Let $g_t$ be a path of hyperbolic isometries,  leaving  invariant the horospheres of same center $c$ as $(P,G)$ and such that $g_0=g$. Then  $\vec{g}:=dg^{-1}\frac{\partial}{\partial t}g_t\vert_{t=0}$.  In particular a parabolic Killing field  is a  Killing field of $\mathbb{H}^3$ tangent to all the horospheres of center $c$, and its restriction to each horosphere gives an Euclidean Killing field.
A convex parabolic polyhedron  is \emph{parabolic infinitesimally  rigid}  if all its parabolic deformations are trivial. 

Equation (\ref{eq:equivariant field}) arises naturally under the point of view of  isometric immersions of  surfaces in the hyperbolic space.
Due to Lemma \ref{lem:CuspToPol} each convex parabolic polyhedral surface $(\partial P,G)$ is given by a  pair $(\phi, \rho)$, where $\rho$ is a cocompact representation of $\pi_1(\Tor)$ in a group of parabolic isometries of $\mathbb{H}^3$ and $\phi$ is a convex polyhedral isometric immersion of the universal cover of $(\Tor,g)$ in the hyperbolic space,  equivariant under the action of $\rho$:  for $x\in\mathbb{R}^2$ and $\gamma\in\pi_1(\Tor)$,
$$\phi(\gamma x)=\rho(\gamma)(\phi(x)).$$

If we derivate a path  $(\phi_t, \rho_t)$  of such pairs with respect to $t$, the property of equivariance above leads to Equation (\ref{eq:equivariant field}), where $Z$ arises from the derivative of $\phi_t$, and $\vec{g}$ from the derivative of $\rho_t(\gamma)$ (here $g=\rho_0(\gamma)$). See \cite{Fillastre2} for analogous considerations.

At each vertex $i$ of a convex parabolic polyhedron $(P,G)$ with center $c$ we can decompose $T_i\mathbb{H}^3$  into a vertical direction that is the direction given by the derivative at $i$ of the ray $ci$, and a horizontal plane, which is orthogonal to the vertical direction. For a vector field $V$ we denote by $V_v$ its vertical component and by $V_h$ its horizontal component. If $Z$ is a parabolic deformation we have by definition: 
\begin{equation}\label{eqn:comp vert}
Z_v(g(i))=dg.Z_v(i).
\end{equation}

Roughly speaking, the proof of Theorem \ref{infinitesimal rigidity} goes as follows. The \emph{radii} of a convex parabolic polyhedron are the particle lenghts of the corresponding cusp. It follows that each convex parabolic polyhedron is defined by its radii. Hence each  parabolic deformation  corresponds to a first-order deformation $\dot{r}=(\dot{r}_1,\ldots,\dot{r}_n)$ of the radii of the polyhedron such that the corresponding  first-order deformation of the singular curvatures $\dot{\kappa}=(\dot{\kappa}_1,\ldots,\dot{\kappa}_n)$ is zero: $\dot{r}$ belongs to the nullspace of the Hessian of the total scalar curvature of the cusp. But we know that in this case the nullspace is reduced to a trivial deformation.
Such proofs have already been used in  cases where no group acts on the deformation \cite{schconnelly,Izmestiev}.

 We denote by $\Delta (M)$ the Hessian of the total scalar curvature $S$ at the point $M$.

\begin{lemma}
Let $Z$ be a parabolic deformation of $(\partial P,G)$. Then $Z_v$ induces a first-order deformation $\dot{r}$ of the radii of $(P,G)$
such that $\dot{r}$ belongs to the nullspace of $\Delta (P/G)$.
\end{lemma}

\begin{proof}
By definition $Z_v$ gives a first-order deformation $\dot{r}$ of the radii of the vertices of $P$. Moreover by  Equation (\ref{eqn:comp vert}) this deformation is well-defined on $P/G$. As $Z_v$ is an infinitesimal deformation of a polyhedral surface of the hyperbolic space, the angles around the particles of $P/G$ must remain equal to $2\pi$ under the deformation, that means that $\dot{\kappa}=0$, hence  $\dot{r}$ belongs to  $\Delta (P/G)$ by (\ref{eqn:dSdh}).
\end{proof}

The converse holds:

\begin{lemma}
Let $\dot{r}\in \Delta (P/G)$. Then there exists a unique parabolic deformation $Z$ of $(\partial P,G)$ such that $\dot{r}(i)=\vert Z_v(i)\vert$.
\end{lemma}
\begin{proof}
In a fundamental domain on $\partial P$ for the action of $G$, for each vertex $i$, we define a vector  $Z_v(i)$ as the unique vertical vector at  $i$ which has norm and direction given by $\dot{r}_i$. We  define this vector for the other vertices  of  $P$ using the action of $G$,   that defines a vector field $Z_v$ on $\partial P$.

The deformation $\dot{r}$ also acts on the projection of the vertices onto a horosphere $H$ of same center $c$ than $(P,G)$. We consider horoprisms given by $H$ together with a triangulation of the faces of $\partial P$. 
Into each horoprism the deformation $\dot{r}$   gives a horizontal deformation that we call $Z_h$. We can extend $Z_h$ to $\partial P$ by gluing the horoprisms. The vector field $Z_h$ is well-defined because  $\dot{r}\in\Delta (P/G)$. Actually we make a little abuse of notation, as the vector $Z_h$ at a vertex $i$ should be defined as the image under the differential of the orthogonal projection from $H$ to the horosphere concentric to $H$ and passing by $i$ of the vector defined on $H$ that we also denote by $Z_h$.

We define the vector field $Z:=Z_v+Z_h$ on $\partial P$. It is an  infinitesimal isometric deformation as it corresponds to a first-order deformations of the particle lengths for a non-varying boundary metric.

It remains to prove that $Z_h$  verifies (\ref{eq:equivariant field}). It will follow that $Z$ verifies (\ref{eq:equivariant field})  because of (\ref{eqn:comp vert}). Consider a fundamental domain on $\partial P$ for the action of $G$. Its projection onto $H$ defines a lattice on $H$ (and hence on $\mathbb{R}^2$). Applying $Z_h$ on this fundamental domain  leads to a first-order deformation of the lattice on $H$. Consider  a  generator $g$ of $G$ given by an edge of the lattice. Let $i$ be a vertex of the lattice (up to project onto $H$). 
Up to compose by a global Killing field suppose that we have $Z_h(i)=0$.
At the vertex $g(i)$, there will be a horizontal first-order displacement, which is given by the deformation of the lattice. Then $Z_h(g(i))$ is the restriction to $g(i)$ of a  Euclidean Killing field.
\end{proof}

\begin{proof}[Proof of Theorem \ref{infinitesimal rigidity}]
Let $Z$ be a parabolic deformation of a convex parabolic polyhedron $(P,G)$. By the lemmas above $Z$ corresponds to a vector $\dot{r}$ of the nullspace of $\Delta (P/G)$, and by Lemma  \ref{lem:HessCusp} we know that  this nullspace  is reduced to the trivial deformation $\mathbb{1}$.
\end{proof}

A hyperbolic
metric $m$ of a convex polyhedral cusp $M$ can be defined as a section of the bundle of scalar products over $M$. A first-order deformation  $\dot{m}$ of  $m$ can be defined as a section of the bundle of symmetric bilinear forms over $M$. Such a deformation is \emph{trivial} if it is given by  the
Lie derivative of $g$ under the action of a vector field of $M$. We will only consider deformations
 $\dot{m}$ such that the metric remains hyperbolic, \emph{i.e.} the first-order variation of the sectional
curvature of $m$ induced by $\dot{m}$ vanishes.  The deformation $\dot{m}$ is \emph{trivial on $\partial M$} if
its restriction to $\partial M$ is zero.

\begin{definition}
A convex polyhedral cusp $M$ is called \emph{infinitesimally rigid} if every  deformation $\dot{m}$ of $M$ which is trivial on $\partial M$ is trivial on the whole $M$.
\end{definition}

\begin{proof}[Proof of Theorem \ref{InfRigCusp}]
It is known that a first-order deformation of a hyperbolic manifold with convex boundary which is trivial on the boundary is equivalent to an equivariant infinitesimal isometric deformation  of the image of the boundary by the developing map. Moreover one is trivial when the other is, see \emph{e.g.} \cite{Schconvex}. It follows that Theorem \ref{InfRigCusp} is equivalent to Theorem \ref{infinitesimal rigidity}.
\end{proof}

\begin{proof}[Proof of Theorem \ref{thm:weak infinitesimal rigidity}]
We sketch the proof as it is word by word the same as in the Section 4 of \cite{schweakconvex}, even if this reference concerns Euclidean polytopes.  The idea of the proof is the same as for convex parabolic polyhedra above: it is sufficient to prove that the matrix $\Delta$ associated to each weakly convex star-shaped  parabolic polyhedra  is negatively definite (the definition of the total scalar curvature doesn't use the convexity). And this follows directly from the case of convex parabolic polyhedra as:
\begin{enumerate}
 \item a weakly convex star-shaped  parabolic polyhedron $P$ is obtained from a convex parabolic polyhedra $P'$  by removing a finite number of simplices (in a fundamental domain) \cite[Lemma 4.1]{schweakconvex};
\item each time we remove a simplex,  the matrix $ \Delta'$ associated to $P'$ changes by the addition of a negatively semidefinite matrix \cite[Lemma 4.3 and 4.4]{schweakconvex};
\item  at the end $\Delta$ is negatively definite as $\Delta'$ is \cite[Lemma 4.5]{schweakconvex}. 
\end{enumerate}
\end{proof}

\bibliographystyle{alpha}
\bibliography{MetricTorusbiblio}
\end{document}